\documentclass[10pt,a4paper]{article}

\usepackage{url}
\usepackage{enumitem}
\usepackage{amsmath, amscd, amssymb, amsthm, latexsym,dsfont}
\usepackage{hyperref, textcomp, xcolor}
\usepackage{color}\usepackage{ulem}
\usepackage{url}\usepackage[utf8]{inputenc}
\usepackage[all]{xy}
\usepackage{forest}
\usepackage{tikz}

\usepackage{ulem}

\usepackage{fancybox}

\newtheorem{theorem}{Theorem}[section]
\newtheorem*{theorem*}{Theorem}
\newtheorem{proposition}[theorem]{Proposition}
\newtheorem{lemma}[theorem]{Lemma}
\newtheorem{corollary}[theorem]{Corollary}
\newtheorem{definition}[theorem]{Definition}
\theoremstyle{definition}
\newtheorem{remark}[theorem]{Remark}
\newtheorem{example}[theorem]{Example}
\newtheorem{notation}[theorem]{Notation}

\if 35
\numberwithin{equation}{section}

\theoremstyle{plain}
\newtheorem{theorem}{Theorem}[section]
\newtheorem{lemma}[theorem]{Lemma}
\newtheorem{proposition}[theorem]{Proposition}
\newtheorem{corollary}[theorem]{Corollary}

\theoremstyle{definition}
\newtheorem{definition}[theorem]{Definition}

\fi

\newcommand\+[1]{\mathcal{#1}}
\newcommand\FSB[2] { {\psi}_{#1\rightarrow #2}}

\newcommand\GS[2] { {\gamma}_{#1\rightarrow #2}}

\newcommand{\vl}{{\blacktriangleleft}}
\newcommand{\vr}{{\blacktriangleright}}
\newcommand{\N}{{\mathbb N}}     

\begin{document}

\title{The algebra of complete binary trees is affine complete}
\if 35
\author[A. Arnold]{Andr\'e { Arnold}
\address{LABRI, UMR 5800,  Universit\'e  Bordeaux, France.}
\email{andre.arnold@club-internet.fr}

\author{Patrick { C\'egielski}\textsuperscript{2}}
\address{LACL, EA 4219, 
Universit\'e  Paris XII -- IUT de S\'enart-Fontainebleau, France.}
\email{cegielski@u-pec.fr }

\author {Serge { Grigorieff}\textsuperscript{3}}
\address{IRIF, UMR 8243, CNRS \& Universit\'e Paris-Diderot, France.}
\email{seg@irif.fr}

\corrauthor[I. Guessarian]{Ir\`ene { Guessarian}}
\address{IRIF, UMR 8243, CNRS \& Sorbonne Universit\'e, France.}
\email{ig@irif.fr}

\dedicatory{ To the memory of Kate Karagueuzian-Gibbons and Giliane Arnold}
\fi

\author{\renewcommand\thefootnote{\arabic{footnote}}
{Andr\'e { Arnold}\textsuperscript{1}}\quad 
{Patrick { C\'egielski}\textsuperscript{2}}\quad 
{Serge { Grigorieff}\textsuperscript{3}}\\
 {Ir\`ene { Guessarian}\textsuperscript{4}}}

\footnotetext[1]{LABRI, UMR 5800,  Universit\'e  Bordeaux, France.}

\footnotetext[2]{LACL, EA 4219, Universit\'e  Paris XII -- IUT de S\'enart-Fontainebleau, France.}

\footnotetext[3]{IRIF, UMR 8243, CNRS \& Universit\'e Paris-Diderot, France.}
\footnotetext[4]{IRIF, UMR 8243, CNRS \& Sorbonne Universit\'e, France.}
\bibliographystyle{plain}

\maketitle

\begin{center}
To the memory of Kate Karagueuzian-Gibbons and Giliane Arnold
\end{center}

\vspace{-0.3cm}
\begin{abstract}
A function on an algebra  is congruence preserving
if, for any congruence, it maps pairs of congruent elements onto pairs of congruent elements.
We show that on the algebra of complete binary trees  
whose leaves are labeled
by letters of an alphabet containing at least three  letters
a function  is  congruence preserving  
 if and only if it is polynomial.
This exhibits an example of  a non commutative and non associative
affine complete algebra. As far as we know, 
it is the first  example of such an algebra.
\end{abstract}
\maketitle

\vspace{-0.5cm}
\section{Introduction}
A function on an algebra  is congruence preserving
if, for any congruence, it maps pairs of congruent elements onto pairs of congruent elements.
Such functions were introduced in Gr\"atzer  \cite{{gratzerUnivAlg}},  
where they are said to have the ``substitution property.''

A polynomial function on an algebra is a function defined by a term of the algebra using variables, constants and the operations of the algebra.
Obviously, every polynomial function is  congruence preserving. 
In most algebras this inclusion is strict.
A very simple example where the inclusion is strict is the additive algebra of natural integers 
$\langle \N,+\rangle$,  cf. \cite{cgg15}.
Up to the example studied in \cite{cgg}, 
all affine complete algebras studied so far were commutative and associative, see 
\cite{gratzer1962,nobauer1978,gratzer1964,plos}.  
The example in \cite{cgg} is  the free monoid
on an alphabet with at least three letters:  its operation is non commutative but associative.
The present paper is a follow-up of  \cite{cgg} though it does not depend on it.
We here prove that
the free algebra with one binary operation and at least three generators is affine complete, which gives a nontrivial example 
of a non associative and non commutative affine complete algebra.

\vspace{-0.5cm}
\section{Preliminary definitions}
\vspace{-0.2cm}
\subsection{Congruences}
\begin{definition}\label{def1:congA} Let $\+A=\langle A\,,\, *\rangle$ be an algebra equipped with a binary operation $*$.
A {\em congruence} $\sim$ on  $\+ A$ is an equivalence relation on $A$ wich is compatible with the operation $*$, 
 i.e.,
for all $a_i,a'_i\in  A$, if
$a_i\sim a'_i$ for $i=1,2$ then $a_1 *a_2 \sim a'_1* a'_2$.
\end{definition}
\begin{lemma}\label{l:ker} 
Let $\+A=\langle A\,,\, \star\rangle$,  $\+B=\langle B\,,\, *\rangle$ be two algebras with  binary operations $\star$ and $*$, 
and $\theta\colon A\to B$ a homomorphism. 
Then the kernel of $\theta$
defines a congruence $\sim_\theta$ on $A$ by $x\sim_\theta y$ iff $\theta(x)=\theta(y)$.
\end{lemma}

\vspace{-0.2cm}
\subsection{Monoids}
\begin{definition} 
Let $\Sigma$ be a nonempty alphabet whose elements are called letters. 
The {\em free monoid}  generated by $\Sigma$ 
is the  algebra $\langle\Sigma^*, \cdot\rangle$. 
Its elements are the finite sequences (or {\em words}) of  elements from $\Sigma$. It is
endowed with the concatenation operation and the unit element is the empty word denoted by $\varepsilon$. 
The free monoid will be abbreviated in $\Sigma^*$ in the sequel.
 
The {\em length} of a word $w\in\Sigma^*$ is the total number of occurrences of letters in $w$ and it is denoted $|w|$.
The  number of occurrences of a letter $b$ in a word $w$ is denoted by $|w|_b$. As usual $ \Sigma^+$ denotes the set $ \Sigma^*\setminus\{\varepsilon\}$.
\end{definition}

\begin{definition}   Let $\Gamma$ be a subset of $\Sigma$. The {\em projection} $\pi_\Gamma$ is the homomorphism
  $\Sigma^*\to \Gamma^*$ which erases all letters not in $\Gamma$ and leaves those in $\Gamma$
  unchanged. 
\end{definition}

 By Lemma \ref{l:ker} the relation $\pi_\Gamma(x)=\pi_\Gamma(y)$ is a congruence. 
We shall use the following homomorphisms on $\Sigma^*$.

\begin{definition}
Let $a\in \Sigma$ and $u\in \Sigma^*$. Then  the {\em substitution}
$\psi_{a\to u}$ is the homomorphism
 $\Sigma^*\to\Sigma^*$ which maps the letter $a$ onto the word $u$ and leaves other letters unchanged.
\end{definition}
%
\vspace{-0.5cm}
\section{Complete binary trees and their congruences}
\vspace{-0.2cm}
\subsection{Complete binary trees}

Let $\Sigma$ be  an alphabet, 
 let  $\Xi=\{\; \vl , \bullet\; ,\vr \;\}$  be   an alphabet disjoint from $\Sigma$
and let $\Theta=\Sigma\cup\Xi$.
 We shall use the interpretation of the free binary algebra 
with generators $\Sigma$ 
as a set of words on the alphabet $\Theta$.

\begin{lemma}\label{l:BTLL}
The free binary algebra with $\Sigma$ as set of generators,
can be seen as the algebra of complete binary trees 
with leaves labeled by letters of \ $\Sigma$
 endowed  with the operation which concatenates two trees 
as the left and right subtrees of a new root. 
It  can also 
be isomorphically
represented by the algebra $\+ {B}=\langle \+T,\star\rangle$  
defined as follows:
\begin{itemize}
\item Its carrier set $ \+ T $ is the least set of non empty words of 
$\Theta ^+$  also called ``trees'', inductively defined by
\begin{enumerate}
\item each letter $a$ in $\Sigma$ is a tree $a$ in $\+ T$
\item if  $t$ and $t'$ are  trees in $\+ T$,  then  the word $\vl \,t \bullet t'\vr$  is a  tree in $\+ T$
\end{enumerate}
\item The binary product operation $\star$ 
defined by: $t\star t' = \vl\, t\bullet t'\vr$
\end{itemize}
\end{lemma}
 This product is neither commutative nor associative. 
The elements of  $\+T$ can be viewed as trees with  leaves  labeled  by letters in the alphabet $\Sigma$. See Figure~\ref{fig:tree}.

\begin{definition}
The set  $ \+ S$ of {\em skeletons}  is the least set of words of $\Xi^*$ 
inductively defined by
 (1) $\varepsilon$ is a skeleton,  and  (2)
 if  $s$ and $s'$ are skeletons, then $\vl \,s \bullet s'\vr $  is a skeleton.

\begin{itemize}
\item[--] The {\em skeleton}   of a tree $t$ is  the word $\sigma(t)= \pi_\Xi(t)\in \Xi^*$
 obtained by erasing all letters not in $\Xi$.
  
\item[--] The {\em foliage}  of a tree $t$
  in $\+ T$ is the word $\varphi(t)=\pi_\Sigma(t)\in \Sigma^+$
  obtained by erasing all letters not in $\Sigma$.

  \end{itemize}
\end{definition}

Note that the skeleton of a tree indeed belongs to $\+S$.

\begin{center}
\begin{figure}[h]
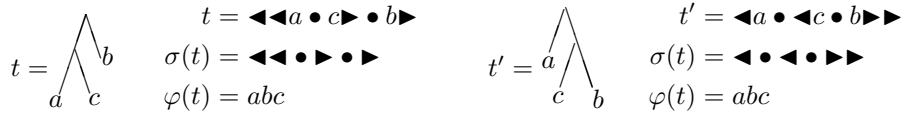

\raisebox{0.5cm}{\hbox{
$t= $\put(4,-7){\line(1,3){10}}\put(0,-12){\shortstack{$a$ }}\put(10,10){\line(1,-3){5.5}} \put(15,-11){\shortstack{$c$ }}
\put(14,23){\line(1,-3){5.5}}\put(20,4){\shortstack{$b$ }}
 }}\qquad\qquad
\begin{minipage}[b]{1.4in}
\begin{align*}  
 t&=\vl\vl a\bullet c\vr\bullet b\vr\\
 \sigma(t)&=\vl\vl\bullet \vr\bullet \vr\\
 \varphi(t)&=abc
\end{align*}
\end{minipage}
\qquad
\raisebox{0.5cm}{\hbox{$t'= $}}
\put(12,40){\line(1,-3){10}}\put(3,17){\shortstack{$a$ }}
\put(10,10){\line(1,3){5.4}} \put(7,4){\shortstack{$c$ }}
\put(6,23){\line(1,3){5.4}}\put(22,2){\shortstack{$b$ }}
\qquad\qquad
\begin{minipage}[b]{1.4in}
\begin{align*}  
 t'&=\vl a\bullet \vl c\bullet b\vr\vr\\
 \sigma(t)&=\vl\bullet\vl\bullet \vr\vr\\
 \varphi(t)&=abc
\end{align*}
\end{minipage}
 \caption{A graphic representation of two trees }\label{fig:tree}
\end{figure}
\end{center}

\if 35
\begin{center}
\begin{figure}[h]
\quad $t$= \raisebox{-1cm}{\begin{forest}
    [ [[$a$][$c$]]
     [$b$]
     ]
 \end{forest}
 \ \ 
 \begin{minipage}[b]{1.4in}
\begin{align*}  
 t&=\vl\vl a\bullet c\vr\bullet b\vr\\
 \sigma(t)&=\vl\vl\bullet \vr\bullet \vr\\
 \varphi(t)&=abc
\end{align*}
\end{minipage}
}
    \qquad\quad
$t'$= \raisebox{-1cm}{\begin{forest}
     [        [$a$ ]   
     [ [$c$] 
       [$b$ ]
       ]
       ]
\end{forest}\ \ 
\begin{minipage}[b]{1.4in}
\begin{align*}  
 t'&=\vl a\bullet \vl c\bullet b\vr\vr\\
 \sigma(t)&=\vl\bullet\vl\bullet \vr\vr\\
 \varphi(t)&=abc
\end{align*}

\end{minipage}
}
 \caption{A graphic representation of two trees }\label{fig:tree}
\end{figure}
\end{center}  
\fi

\begin{proposition}\label{p:homo} For all $t, t'\in \+ T$,

  (1) $\sigma(t\star t')= \vl \sigma(t)\bullet\sigma(t') \vr$,

  (2) $\varphi(t\star t')=  \varphi(t)\varphi(t')$,

(3) $|\sigma(t)|=3|\varphi(t)|-3$.
\end{proposition}

\begin{proof}
Point (3) is the variant (due to the extra symbols $\vl , \bullet\; ,\vr$) of the classical result that a 
complete binary tree has one more leaf than it has  nodes.
\end{proof}

\begin{proposition} \label{p:sf} 
 (1)  Let $u\in \Sigma^+ $ and $s\in \+ S$ such that $|s|=3|u|-3$. Then there exists a
  unique tree $t=\tau(u,s)$ 
 with foliage $\varphi(t)=u$ and skeleton $\sigma(t)= s$.
  
  (2) If $t$ and $t'$ are such that $\varphi(t)=\varphi(t')$ and $\sigma(t)= \sigma(t')$, then $t=t'$.
\end{proposition}
\begin{proof} The proof is by induction on $|u|$.
  If $|u|=1$ then $u=a$, $s=\varepsilon$ and $\tau(a, \varepsilon)=a$.
  If $|u|>1$, there exists $u_1,u_2 \in \Sigma^+, s_1,s_2 \in \+ S$, such that
 $ u=u_1u_2$, $s =\vl s_1\bullet s_2 \vr$  and $|s_i|= 3|u_i|-3$. Hence $\tau(u,s)=
 \tau(u_1,s_1)\star \tau(u_2,s_2)$.
 
 (2) immediately follows from (1).
\end{proof}

\vspace{-0.2cm}
\subsection{Congruences on complete binary trees}
\begin{example}\label{e:cong} 
We shall use later two congruences  defined as kernels of homomorphisms. (1) Equality of skeletons: $t\sim_\sigma t'$ iff  $\sigma(t)=\sigma(t')$. (2) Equality of foliages: $t\sim_\varphi t'$ iff  $\varphi(t)=\varphi(t')$.
\end{example}

Other fundamental congruences are the kernels of the grafting homomorphisms  defined below.
\begin{definition}\label{grafting} 
Let $a\in \Sigma$ and $\tau\in\+ T$.  
Then  ${\GS{a}{\tau}}: \+ T \to \+ T$
 is the homomorphism on the free algebra of trees such that,
for $b\in \Sigma$, the tree ${\GS{a}{\tau}}(b)$ is equal to $\tau$  if $b=a$, and to
  $b$ otherwise.
\end{definition}
\penalty 10000
The following Proposition and Lemma are easily proved by induction on $t$.
\begin{proposition}\label{gtopsi}
For all $\tau, t\in \+ T, a\in \Sigma$, $\varphi(\GS{a}{\tau} (t))= {\FSB{a}{\varphi(\tau)}}(\varphi(t)))$,   i.e., 
the following diagram is commutative: {\small
\[
\begin{CD}
\+T @>\GS{a}{\tau}>> \+T\\
@VV{\varphi}V @VV{\varphi}V\\
\Sigma^* @>\FSB{a}{\varphi(\tau)}>> \Sigma^*     
\end{CD}
\]
}
\end{proposition}

\begin{lemma}\label{d:idempotent}  A grafting  $\GS{a}{\tau} $ is {\em idempotent},
 i.e., $\GS{a}{\tau} \circ \GS{a}{\tau}= \GS{a}{\tau}$,
 if and only if the letter $a$ does not  appear in the foliage $\varphi(\tau)$
\end{lemma}
%
\vspace{-0.5cm}
\section{Congruence preserving functions on  trees}
\label{sec:congr-pres-funct}

We now study congruence preserving functions on the algebra $\langle\+T,\star\rangle$. From now on, $f,\; g$ will be congruence preserving functions on $\+T$.
\begin{definition}\label{def1:cp}
A function $f\colon  \+T \to  \+T$ is {\em congruence preserving}  (abbreviated into CP) if for  all  congruences  $\sim$  on~$\+ T$, for all $t,t'$ in $\+ T$, $t\sim t'\  \Longrightarrow \ f(t)\sim f(t')$.
\end{definition}

We start with a very convenient result.
\begin{proposition}\label{p:G2} Let  $\gamma_a= {\GS{a}{\tau}}$ and $\gamma_b= {\GS{b}{\tau}}$ be  two graftings
 with $a\neq b$. For any $t,t'$ if   $\gamma_a(t) =  \gamma_a(t')$ and $\gamma_b(t) =  \gamma_b(t')$ then $t=t'$
\end{proposition}
\begin{proof}
  By induction on the set of pairs $(t,t')$ ordered by; $(t_1,t_2)\leq (t'_1,t'_2)$ 
  iff $|t_i|\leq |t'_i|$ for $i=1,2$. 
  
  {\em Basis} 
{\bf (i) }{\em One of $t,t'$ is a letter $c\in\{a,b\}$.}
 Say $t=a$. 
 Then $\gamma_b(t) = a$ hence $\gamma_b(t')=a$.
 If $b$ does not occur in $t'$ then $\gamma_b(t')=t'$ 
 and equality $\gamma_b(t')=a$ implies $t'=a$ hence $t=t'$,
 as wanted.
 If $b$ does occur in $t'$ then equality $\gamma_b(t')=a$ implies
 $\tau=a$ and $t'=b$. In particular, $\gamma_a(t')=b$ which is in contradiction 
 with $\gamma_a(t')=\gamma_a(t)$ since $\gamma_a(t)=\gamma_a(a)=\tau=a$.
 
{\bf (ii) } {\em One of $t,t'$ is a letter $c\notin\{a,b\}$.}
 Then  $\gamma_a(c) =  \gamma_b(c)=c$, hence $\gamma_a(t') =  \gamma_b(t')=c$ implying $t'=c$ and $t=t'$ as wanted. 

  {\em Induction} Otherwise, we have $t=t_1\star t_2$ and $t'=t'_1\star t'_2$ with
  $(t_i,t'_i)< (t,t')$. hence $\gamma_a(t_1) \star  \gamma_a(t_2)=\gamma_a(t'_1) \star  \gamma_a(t'_2)$ implying $\gamma_a(t_i) =\gamma_a(t'_i) $. Similarly 
 $\gamma_b(t_i) =\gamma_b(t'_i) $. By the induction $t_i=t'_i$  hence $t=t'$.
  \end{proof}
\begin{proposition}\label{p:idG}
If $f$ is CP then 
for every idempotent grafting 
${\GS{a}{t}}$,  we have   $\GS{a}{t}(f(a)) =  \GS{a}{t}(f(t))$.
\end{proposition}

\begin{proof}
  As $\GS{a}{t}$ is idempotent we have   
$\GS{a}{t}(a) = \GS{a}{t}(\GS{a}{t}(a))$. 
Now, $\GS{a}{t}(a)=t$ hence $\GS{a}{t}(a) = \GS{a}{t}(t)$. 
Since $\ker(\GS{a}{t})$  is a   congruence and $f$ is CP,  
we conclude $\GS{a}{t}(f(a) )= \GS{a}{t}(f(t))$.
\end{proof}
\begin{corollary}\label{c:idG} If   $f(a)=g(a)$, then  for any idempotent grafting $\GS{a}{t}$, we have  $\GS{a}{t}(f(t))=\GS{a}{t}(g(t))$.
\end{corollary}
\begin{proof}By Proposition \ref{p:idG}
we have,  $\GS{a}{t}(f(t)) =\GS{a}{t}(f(a))$
and $\GS{a}{t}(g(a))=  \GS{a}{t}(g(t))$.
As  $f(a)=g(a)$, we infer $\GS{a}{t}(f(t))=\GS{a}{t}(g(t))$.
\end{proof}
Proposition~\ref{p:idG} and its Corollary tell us that the knowledge of $f(a)$ for all $a\in \Sigma$
gives a lot of information about the value of $f$ on $\+ T$.
The following theorem shows that, in fact,
$f$ is completely determined by its value on $\Sigma$.

\begin{theorem}\label{t:ppal}  Suppose $\Sigma$ has at least three letters,
if $f$ and $g$ are CP  functions on $\+T$ 
  such that for all $a \in  \Sigma$, $f(a)=g(a)$  then for all $t\in\+ T$, $f(t)=g(t)$.
\end{theorem} 
\begin{proof}

Let $t\notin \Sigma$.
  The proof depends on the number $N(t)$
  of letters of $\Sigma$ which do not appear  in the foliage $\varphi(t)$ of $t$.

{\bf 1. Case $N(t)>1$}\quad
Let $a,b$ be two letters which do not occur in the foliage $\varphi(t)$ of $t$.
  Graftings ${\GS{a}{t}}$ and  ${\GS{b}{t}}$ are idempotent.
  By   Corollary \ref{c:idG}
    we have $\GS{c}{t}(f(t)) =\GS{c}{t}(g(t))$ for $c=a,\;b$,  
and Proposition \ref{p:G2} yields $g(t)=f(t)$.

{\bf 2. Case  { $N(t)\leq 1$}}\quad
Let $c$ be any letter  and 
  let $t_c$ be the tree obtained by substituting $c$ to all letters in $t$. Then
  $N(t)=|\Sigma|-1\geq 2$, and thus $g(t_c)=f(t_c)$.
The trees $t$ and $t_c$ obviously have the same skeleton
hence  $t\sim_\sigma t'$
(cf.  Example \ref{e:cong}  (1)).  As $f$ and $g$ are congruence preserving, we also have $f(t)\sim_\sigma f(t_c) = g(t_c)\sim_\sigma  g(t)$.
As $f(t)$ and $g(t)$ have the same skeleton,  by Proposition~\ref{p:sf}~(2), we know that $f(t)=g(t)$ if and only
if $\varphi(f(t))=\varphi(g(t))$.

Assume by way of contradiction, $\varphi(f(t))\neq\varphi(g(t))$.  By Proposition~\ref{p:homo}, as  $f(t)$ and $g(t)$ have the same skeleton, their foliages have the same length, hence
 $\varphi(f(t))=wau$ and $\varphi(g(t))=wbv$ with $a,b\in\Sigma$, $x\neq y$.

{\em Subcase $N(t)=1$}
Let $c$ be the unique letter which does not appear in $t$.
 Then as ${\GS{c}{t}}$ is idempotent (cf. Lemma \ref{d:idempotent}) we
      have, by Corollary \ref{c:idG},
      ${\GS{c}{t}}(f(t))={\GS{c}{t}}(g(t))$.  By Proposition~\ref{gtopsi},
${\FSB{c}{\varphi(t)}}(\varphi(f(t)))={\FSB{c}{\varphi(t)}}(\varphi(g(t)))$, which implies
${\FSB{c}{\varphi(t)}}(au)={\FSB{c}{\varphi(t)}}(bv)$. This is possible only if one of the two
letters $a,b$ is $c$. But in this case the other letter is the first letter of $\varphi(t)$ which
 cannot be $c$ by the choice of $c$, a contradiction.

\medskip
{\em Subcase $N(t)=0$}\quad 
Since $|\Sigma|\geq 3$ there exists a letter $c\notin\{a,b\}$.  Then ${\GS{c}{t}}$ is idempotent. Let $t'=
{\GS{a}{c}}(t)$.
As $N(t')=1$, we have $f(t') = g(t')$.

But $t$ and $t'$ are congruent for the congruence $\ker({\GS{a}{c}})$, and as $f$ is CP, we  also have
${\GS{a}{c}}(f(t)) =
{\GS{a}{c}}(f(t'))$. Similarly  $ {\GS{a}{c}}(g(t')) = {\GS{a}{c}}(g(t))$. As 
$f(t')=g(t')$, we infer  ${\GS{c}{a}}(f(t)) ={\GS{c}{a}}(g(t))$, implying
 ${ \FSB{a}{c}} (\varphi(f(t))) = { \FSB{a}{c}} (\varphi(g(t)))$: as
$\varphi(f(t))=wau$ and $\varphi (g(t))=wbv$,
we must  also  have $ c = { \FSB{a}{c}}(a) = { \FSB{a}{c}}(b) = b $,  a contradiction.
\end{proof}
%

  \section{The algebra of binary trees is affine complete}
  Throughout this section, $f$ is a fixed CP function on $\+T$.
\vspace{-0.2cm}
\subsection{Polynomials on trees}
\begin{definition} Let $x\not\in\Sigma$ be a variable. A {\em polynomial} $T(x)$  is a tree on the alphabet $\Sigma\cup\{x\}$.
\end{definition}

With every polynomial $T(x)$ we  associate a {\em polynomial function}  $T\colon\+ T\to\+T$ defined by $T(t)={\GS{x}{t}} (T(x))$.
Obviously,  every polynomial
function is CP.

This section is devoted to proving the converse  which amounts to saying that
the algebra $\langle \+ T , * \rangle$ is affine complete.
\begin{theorem}\label{t:CP=aff}
Every   CP function is polynomial.
\end{theorem}

By theorem~\ref{t:ppal}, a CP function $f$  is  polynomial if there exists a
polynomial  $T_f(x)$ such that for all $a\in \Sigma$, $f(a)=T_f(a)$. Hence to prove Theorem \ref{t:CP=aff}, we will construct such a polynomial in the next subsection.

\vspace{-0.2cm}
\subsection{The polynomial associated with a CP function}
As  $\sigma(a)=\varepsilon$, for all $a\in\Sigma$,
 if $f$ is CP then all $f(a)$ have the same skeleton. 
For any pair $a,b\in\Sigma$ with $a\neq b$, we have ${\GS{a}{b}}(a)={\GS{a}{b}}(b)$ and hence, by Lemma \ref{l:ker},
${\GS{a}{b}}(f(a))={\GS{a}{b}}(f(b))$.
Thus, the next proposition can be applied to $f$.
\begin{proposition}\label{p:pol}
  Let $g: \Sigma \to \+ T$ such that (1) all $g(a)$ have the same skeleton $s$ and (2) $\forall a\neq
  b$,
  ${\GS{a}{b}}(g(a))={\GS{a}{b}}(g(b))$.
  Then there exists a polynomial $T_g$ such that $ g(a)=T_g(a)$ for all $a\in\Sigma$.
\end{proposition}
\begin{proof}
  The proof is by induction of the size  of the common skeleton $s$.

  {\em Basis} If $s=\varepsilon$ then each $g(a)$ is a letter $x_a$.
 By Hypothesis (2), we have  ${\GS{a}{b}}(x_a)={\GS{a}{b}}(x_b)$. This last equality can happen when (i)
$x_a=x_b$, or (ii) $\{x_a,x_b\}=\{a,b\}$.

We first show that if there exists  an $a$ such that $x_a=c\neq a$ then $\forall b\neq a$, $x_b=c$.
 First for all  $b \neq c$ we have  either  (i) $x_a=x_b$ or (ii) $ \{a,b\}=\{x_a,x_b\}$ which is
impossible since $x_a=c\notin \{a,b\}$. Hence (i) holds and $x_b=x_a=c$.
If $x_c\neq c=x_a$, we would infer from ${\GS{a}{c}}(g(a))={\GS{a}{c}}(g(c))$ that  $\{a,c\}=\{c,x_c\}$,  hence $x_c=c$ holds also for $c$.

 Otherwise, $\forall a,\  x_a=a$, hence $T_g=x$.

{\em Induction} Each $g(a)$ is equal to $g_1(a)\star g_2(a)$.  It is easy to check that both $g_i$
satify Hypothesis  (1) and (2). Hence $T_g=T_{g_1}\star T_{g_2}$.
\end{proof}

\vspace{-0.5cm}
\section{Conclusion}
We proved that, when $\Sigma$ has at least three letters, 
the algebra $\+ B$ of binary trees with leaves labeled by letters of  $\Sigma$ is a non commutative non associative affine complete algebra. This result   extends to non commutative non associative affine complete algebras with unit by adding a unit element to $\+T$. By forgetting skeletons and replacing graftings ${\GS{a}{\tau}}$ with
 substitutions $\psi_{a\to u}$, the results in Sections 4 and 5
go through mutatis mutandis when $\+ B$ is replaced by the free monoid  $\Sigma^*$ on an alphabet $\Sigma$ with at least three letters. 
This yields a  simpler and shorter proof of the main result of \cite{cgg}, i.e., the affine completeness of  $\Sigma^*$. 

 Whether  similar results might hold if $\Sigma$ has only two letters are open problems.
The use of  $\Sigma$ was essential in the proof that $\+ B$ is  affine complete. We do not know whether algebras of trees without labels would still be affine complete.


\vspace{-0.5cm}

\end{document}